\documentclass[a4,12pt]{article}
\usepackage{amsmath, amssymb, amsfonts, amstext, amsthm, textcomp}
\usepackage[mathscr]{euscript}
\newtheorem{thm}{Theorem}[section]
\newtheorem{lem}{Lemma}[section]

\newtheorem{defn}{Definition}[section]

\begin{document}
\begin{center}
 {\bf \large Pseudo Principal Pivot Transform: The Group Inverse Case}
\vspace{.5cm}

             { \bf Kavita Bisht}\\
Department of Mathematics\\
                         Indian Institute of Technology Madras\\

                               Chennai 600 036, India \\
          
                and\\
              {\bf K.C. Sivakumar} \\
                         Department of Mathematics\\
                         Indian Institute of Technology Madras\\

                               Chennai 600 036, India. \\
\end{center}

\begin{center}
{\bf Abstract}
\end{center}
In this short note, we prove a formula for the group inverse of a block matrix and consider the pseudo principal pivot transform expressed in terms of group inverses. Extensions of the usual principal pivot transform, where the usual inverse is replaced by the Moore-Penrose inverse, were considered in the literature.  The objective here is to derive the corresponding formulae for the group generalized inverse. These are expected to be useful in our future work.


\newpage
\section{Introduction and Preliminaries}
Let $A\in \mathbb{C}^{n\times n}$. If there exists $X\in \mathbb{C}^{n \times n}$ satisfying the equations $AXA=A, ~XAX=X$ and $AX=XA$, then such an $X$ can be shown to be unique. This unique solution is called the {\it group inverse} of $A$ and is denoted by $A^{\#}$. The nomenclature group inverse arises from the fact that the positive powers of $A$ and $A^{\#}$ together with the idempotent matrix $AA^{\#}$ as the identity element, form an abelian group under matrix multiplication and was thus named by I.Erdelyi in $1967$. The group inverse of a matrix $A\in \mathbb{C}^{n \times n}$ need not always exist. It is well known that $A^{\#}$ exists if and only if $R(A) \cap N(A)=\{0\}$. In other words, $A^{\#}$ exists if and only if the index of $A$ equals $1$. Let us recall that the {\it index} of a square matrix $A$ is the smallest positive integer $k$ such that $rank(A^k)=rank(A^{k+1})$. If $A$ is invertible, then the index is defined to be zero. An easy dimensionality argument can be used to show that the 
index exists for any square matrix. Thus, $A^{\#}$ exists if and only if $rank(A)=rank(A^2)$. For more details and other properties of the group inverse, we refer to the book \cite{bens}. 

Let $A,B,C $ and $D \in \mathbb{C}^{n \times n}$ and $M=
\left(\begin{array}{cc} 
A & B \\ 
C & D 
\end{array} \right)$. Suppose that $A^{-1}$ exists. Then the matrix $D-CA^{-1}B$ is called the {\it Schur complement} of $A$ in $M$. Next, let $D^{-1}$ exist. Then the matrix $A-BD^{-1}C$ is called the {\it associated Schur complement} of $D$ in $M$. The usefulness of the Schur complement is well documented in numerous texts on matrices and numerical analysis (see for instance, \cite{zh}). Let us turn our attention to the principal pivot transform. Tucker introduced this notion in his studies on linear programming problems. Let $M$ be the block matrix, defined as above. Suppose that $A^{-1}$ exists. Then the principal pivot transform of $M$ relative to $A$ is defined by the matrix 
\begin{center}
$\left(\begin{array}{cc}
  A^{_1} & -A^{-1}B \\
  CA^{-1}  & D-CA^{-1} B
 \end{array}\right)$.
\end{center}
Note that the bottom right block of the principal pivot transform is the Schur complement of $A$ in $M$, establishging a relationship between these notions. We refer to the recent work \cite{tsat} for some interesting properties of the principal pivot transform. 

Extensions of the formulae for the Schur complement \cite{carlhaymark} and the principal pivot transform
\cite{meen}, where the usual inverse is replaced by the Moore-Penrose inverse were studied rather long ago.  
We refer to these as the {\it pseudo Schur complement} and {\it pseudo principal pivot transform}, respectively. In \cite{kavi}, some of the results of \cite{tsat} and a few inheritance properties were proved, for the pseudo Schur complement and the pseudo principal pivot transform. In the present work, we consider the case of the group inverse, especially for the pseudo principal pivot transform. First, we prove a formula for the group inverse of a block matrix in the presence of some assumptions and then derive the basic properties of the principal pivot transform.

\section{The Pseudo Principal Pivot Transform in terms of the Group Inverse.}
We begin by proving a formula for the group inverse of partitioned matrices. 
\begin{thm}\label{GA}
Let $A,B,C $ and $D \in \mathbb{C}^{n \times n}$. Suppose that $A^{\#}$ exists. Set $K=D-CA^\# B$ and $M=
\left(\begin{array}{cc} 
A & B \\ 
C & D 
\end{array} \right)$. Suppose that $K^{\#}$ exists. Then $ R(C^*)\subseteq R(A^*)$, $R(B)\subseteq R(A)$, $R(C)\subseteq R(K)$ and 
$R(B^*)\subseteq R(K^*)$ if and only if
\begin{center}
$M^\#=\left(\begin{array}{cc}
A^\# +A^\# BK^\# CA^\# & -A^\# BK^\# \\
-K^\# CA^\# & K^\#
\end{array} \right)$ . 
\end{center} 
\end{thm}
\begin{proof}
First, we observe that $CA^\# A=C$ (since $R(C^*)\subseteq R(A^*)$) and $BK^\# K=B$ (since $R(B^*)\subseteq R(K^*)$). We then have 
\begin{eqnarray*}
DK^\# K&=&(K+CA^\# B)K^\# K\\
 &=& K+CA^\# BK^\# K\\
&=& K+CA^\# B\\
&=& D.
\end{eqnarray*}
Also, 
\begin{eqnarray*}
A^{\#} B + A^{\#} BK^{\#} CA^{\#}B-A^{\#}BK^\# D&=&A^\# B + A^\# BK^\# (CA^\# B- D)\\
 &=&A^\# B-A^\#BK^\# K\\
&=&0.
\end{eqnarray*}
\newpage
Set 
\begin{center}
$X=\left(\begin{array}{cc}
  A^\#+A^\# BK^\# CA^\# & -A^\# BK^\# \\
  -K^\# CA^\# & K^\#
 \end{array}\right)$. 
\end{center}
From the earlier calculations one has
\begin{eqnarray*}
 XM & =& \left(\begin{array}{cc}
  A^\# + A^\# BK^\# CA^\# & -A^\# BK^\# \\
  -K^\# CA^\# & K^\#
 \end{array}\right)
\left(\begin{array}{cc}
  A & B \\
  C & D
 \end{array}\right)\\
 & =& \left(\begin{array}{cc}
  A^\# A & 0 \\
  0 & K^\# K
 \end{array}\right)
\end{eqnarray*}
and
\begin{eqnarray*}
MX & = & \left(\begin{array}{cc}
  A & B \\
  C & D
 \end{array}\right)
\left(\begin{array}{cc}
  A^\#+A^\# BK^\# CA^\# & -A^\# BK^\# \\
  -K^\# C A^\# & K^\#
 \end{array}\right) \\
& = & \left(\begin{array}{cc}
  AA^\# & 0 \\
  0 & KK^\#
 \end{array}\right),
\end{eqnarray*} 
where we have used the facts that $AA^\# B=B$ (since $R(B)\subseteq R(A)$) and 
$ C=KK^\# C$ (since $R(C)\subseteq R(K)$). So, $XM=MX$. Also,
\begin{eqnarray*}
MXM&=&
\left(\begin{array}{cc}
  A & BK^\# K \\
  C & DK^\# K
 \end{array}\right)\\
&=&\left(\begin{array}{cc}
A & B\\
C & D
\end{array}\right)\\
&=&M.
\end{eqnarray*}
Further, \\
\begin{eqnarray*}
 XMX 
&=&
\left(\begin{array}{cc}
  A^\# A & 0 \\
  0 & K^\# K
 \end{array}\right)
\left(\begin{array}{cc}
  A^\#+A^\# BK^\# CA^\# & -A^\# BK^\# \\
  -K^\# C A^\# & K^\#
 \end{array}\right) \\
&=&
\left(\begin{array}{cc}
  A^\#+A^\# BK^\# CA^\# & -A^\# BK^\# \\
  -K^\# C A^\# & K^\#
 \end{array}\right)\\
&=& X. 
\end{eqnarray*}
This completes the proof of the necessity part.
\newpage
Conversely, suppose that 
\begin{center}
$M^{\#}=\begin{pmatrix}
A^\# +A^\# BK^\# CA^\# & -A^\# BK^\# \\
-K^\# CA^\# & K^\#
\end{pmatrix}$.
\end{center}
Then, 
\begin{center}
$MM^{\#}=\begin{pmatrix}
AA^{\#}+AA^{\#}BK^{\#}CA^{\#}-BK^{\#}CA^{\#} & -AA^{\#}BK^{\#}+BK^{\#}\\
CA^{\#}+CA^{\#}BK^{\#}CA-DK^{\#}CA^{\#} & -CA^{\#}BK^{\#}+DK^{\#}
\end{pmatrix}$.
\end{center}
Equating the top left blocks of $MM^{\#}M$ and $M$, we then have 
\begin{center}
$AA^{\#}A+AA^{\#}BK^{\#}CA^{\#}A-BK^{\#}CA^{\#}A-AA^{\#}BK^{\#}C+BK^{\#}C=A$.
\end{center}
This simplifies to
\begin{center}
$A+(I-AA^{\#})BK^{\#}C(I-A^{\#}A)=A$,
\end{center}
which in turn yields 
\begin{eqnarray}\label{eq1}
(I-AA^{\#})BK^{\#}C(I-A^{\#}A) &=& 0.
\end{eqnarray}
By equating the top right blocks of $MM^{\#}M$ and $M$, we get 
\begin{center}
$AA^{\#}B+AA^{\#}BK^{\#}CA^{\#}B-BK^{\#}CA^{\#}B-AA^{\#}BK^{\#}D+BK^{\#}D=B$,
\end{center}
so that one has
\begin{center}
$AA^{\#}B+(I-AA^{\#})BK^{\#}K=B$. 
\end{center}
This reduces to
\begin{eqnarray}\label{eq2}
(I-AA^{\#})B(I-K^{\#}K) &=& 0.
\end{eqnarray}
Again, equating the bottom left blocks of $MM^{\#}M$ and $M$,
\begin{center}
$CA^{\#}A+CA^{\#}BK^{\#}CA^{\#}A-DK^{\#}CA^{\#}A-CA^{\#}BK^{\#}C+DK^{\#}C=C,$
\end{center}
we then have
\begin{eqnarray}\label{eq3}
(I-KK^{\#})C(I-A^{\#}A) &=& 0.
\end{eqnarray}
We also have $MM^{\#}=M^{\#}M$ and so, on equating the bottom left blocks of $MM^{\#}$ and $M^{\#}M$, we have 
\begin{eqnarray}\label{eq4}
K^{\#}C(I-A^{\#}A) &=&(I-KK^{\#})CA^{\#}.
\end{eqnarray}
\newpage
Premultiplying $(4)$ by $K$, we get 
\begin{center}
$KK^{\#}C(I-A^{\#}A)=0$
\end{center}
and thus from equation \eqref{eq3}, one has $CA^{\#}A=C$. So, $R(C^*)\subseteq R(A^*)$. Postmultiplying $(4)$ by $A$ and using the fact that $CA^{\#}A=C$, we get
\begin{center}
$(I-KK^{\#})C=0$.
\end{center}
Thus $R(C)\subseteq R(K)$.
Now, equating top right blocks of $MM^{\#}$ and $M^{\#}M$, one has 
\begin{eqnarray}\label{eq5}
(I-AA^{\#})BK^{\#}&=&A^{\#}B(I-KK^{\#}).
\end{eqnarray}
Postmultiplying equation \eqref{eq5} by $K$, we get
\begin{center}
$(I-AA^{\#})BK^{\#}K=0$
\end{center}
and from equation \eqref{eq2}, we get $AA^{\#}B=B$. Thus $R(B)\subseteq R(A)$. Finally, premultiplying \eqref{eq5} by $A$ and using the fact that $AA^{\#}B=B$, 
\begin{center}
$B(I-K^{\#}K)=0$.
\end{center}
Thus $R(B^*)\subseteq R(K^*)$.
\end{proof}

Next, we state a complementary result, whose proof is similar to the proof of Theorem \ref{GA}. Note that this result uses the pseudo Schur complement $L=A-BD^\# C$, which will be called the {\it complementary Schur complement}.

\begin{thm}\label{GD}
Let $M=
\left(\begin{array}{cc}
  A & B \\
  C & D
 \end{array}\right)$ as above. Suppose that $D^{\#}$ exists and $L=A-BD^\# C$. Suppose that $L^{\#}$ exists. Then $R(B^*)\subseteq R(D^*)$, 
$R(C)\subseteq R(D)$, $R(B)\subseteq R(L)$ and $R(C^*)\subseteq R(L^*)$ if and only if
\begin{center}
$M^\#=
 \left(\begin{array}{cc}
  L^\# & -L^\# BD^\# \\
  -D^\# C L^\# & D^\#+D^\# C^\# BD^\#
\end{array}\right)$. 
\end{center}
\end{thm}

Next, we define the principal pivot transform in terms of the group inverse. As mentioned earlier, principal pivot transform involving Moore-penrose inverse of the block matrices was studied in \cite{meen}.

\newpage
\begin{defn}\label{D: D1}  Let $M$ be defined as above and $A^{\#}$ exist. Then the pseudo principal pivot transform of $M$ relative to $A$ is defined by 
\begin{center}
$P:=pppt(M,A)_{\#}=
\left(\begin{array}{cc}
  A^\# & -A^\# B \\
  CA^\#  & K
 \end{array}\right)$,
\end{center} where $K=D-CA^\# B$. Next, suppose that $D^{\#}$ exists. The complementary pseudo principal pivot transform of $M$ relative to $D$ is defined by 
\begin{center}
$Q:=cpppt(M,D)_{\#}=
\left(\begin{array}{cc}
 L & BD^\# \\
-D^\# C & D^\#
\end{array}\right)$, 
\end{center} 
 where $L=A-BD^\# C$. 
\end{defn}

Both the operations of pseudo principal transforms are involutions, in the presence of certain assumptions, as we prove next.

\begin{lem}\label{L1}
Let $M$ be defined as above. Let $A^{\#}$ and $D^{\#}$ exist.\\
$(i)$ Suppose that $R(B)\subseteq R(A)$ and $R(C^*)\subseteq R(A^*)$. Then $pppt\left(P, A^\# \right)_{\#}=M$.\\
$(ii)$ Suppose that $R(C)\subseteq R(D)$ and $R(B^*)\subseteq R(D^*)$. Then $cpppt\left(Q, D^\# \right)_{\#}=M$. 
\end{lem}
\begin{proof}
$(i)$: Set $W=A^\#, X=-A^\# B$ and $Y=CA^\#$. Then $P=
\left(\begin{array}{cc}
W & X\\
Y & K
\end{array}\right)$. So,
\begin{eqnarray*}
pppt(P, A^\#)_{\#}=
pppt (P, W)_{\#}&=&
\left(\begin{array}{cc}
W^\# & -W^\# X\\
YW^\# & K-YW^\# X
\end{array}\right)\\
&=&
\left(\begin{array}{cc}
  A & -A(-A^\# B) \\
  CA^\# A  & D-CA^\# B +CA^\# AA^\# B
 \end{array}\right)\\
&=&\left(\begin{array}{cc}
  A &  B \\
  C & D
 \end{array}\right) \\
&=& M. 
\end{eqnarray*}
$(ii)$: The proof is similar to part $(i)$.
\end{proof}

Finally, we derive the domain-range exchange property. Domain-range exchange property in terms of Moore-Penrose inverse of the matrix is proved in \cite{kavi}. This property is well known in the nonsingular case \cite{tsat}. 

\begin{lem}\label{L2} Let $M$ be defined as earlier. Let $A^{\#}$ and $D^{\#}$ exist.\\
$(i)$ Suppose that $R(B)\subseteq R(A)$ and $R(C^*)\subseteq R(A^*)$.
Then $M$ and $P=pppt(M,A)_{\#}$ are related by the formula: 
\begin{center}
$M
\left(\begin{array}{cc}
x^{1}\\
x^{2}
\end{array}\right)
=
\left(\begin{array}{cc}
AA^\# y^{1} \\
y^{2}
\end{array}\right)$
if and only if $P
\left(\begin{array}{cc}
  y^{1} \\
  x^{2}
 \end{array}\right)
=
\left(\begin{array}{cc}
  A^\# A x^{1} \\
  y^{2}
 \end{array}\right)$. 
\end{center}
(ii) Suppose that $R(C)\subseteq R(D)$ and $R(B^*)\subseteq R(D^*)$. Then $M$ and $Q=cpppt(M,D)_{\#}$ are related by the formula:
\begin{center}
$M
\left(\begin{array}{cc}
x^{1}\\
x^{2}
\end{array}\right)
=
\left(\begin{array}{cc}
y^{1}\\
DD^\# y^{2}
\end{array}\right)$
if and only if
$Q
\left(\begin{array}{cc}
y^{1}\\
x^{2}
\end{array}\right)
=
\left(\begin{array}{cc}
x^{1}\\
D^\# Dy^{2}
\end{array}\right)$.
\end{center} 
\end{lem}
\begin{proof}
We prove (i). The proof for (ii) is similar. Suppose that 
\begin{center}
$M
\left(\begin{array}{cc}
x^{1}\\
x^{2}
\end{array}\right)
=
\left(\begin{array}{cc}
AA^\# y^{1}\\
y^{2}
\end{array}\right)$.
\end{center}
Then 
\begin{center}
$Ax^{1}+Bx^{2}=AA^\# y^{1}$
\end{center}
and 
\begin{center}
$Cx^{1}+Dx^{2}=y^{2}$.
\end{center}
Premultipling the first equation by $A^\#$ (and rearranging) we get 
\begin{center}
$A^\# y^{1}-A^\# Bx^{2}=A^\# Ax^{1}$.
\end{center}
Premultiplying this equation by $C$, we then have 
\begin{center}
$CA^\# y^{1}-CA^\# Bx^{2}=CA^\# Ax^{1}=Cx^{1}$.
\end{center}
So, $CA^\# y^{1}+Kx^{2} =CA^\# y^{1}+Dx^{2}-CA^\# Bx^{2} =Cx^{1}+Dx^{2} =y^{2}$. Thus, 
\begin{center}
$P
\left(\begin{array}{cc}
y^{1}\\
x^{2}
\end{array}\right)
=
\left(\begin{array}{cc}
A^\# y^{1}-A^\# Bx^{2}\\
CA^\# y^{1}+Kx^{2}
\end{array}\right)
=
\left(\begin{array}{cc}
A^\# Ax^{1}\\
y^{2}
\end{array}\right)$.
\end{center}
Conversely, let 
\begin{center}
$P
\left(\begin{array}{cc}
y^{1}\\
x^{2}
\end{array}\right)
=
\left(\begin{array}{cc}
A^\# Ax^{1}\\
y^{2}
\end{array}\right)$. 
\end{center}
Then 
\begin{center}
$A^\# y^{1}-A^\# Bx^{2}=A^\# Ax^{1}$
\end{center}
and 
\begin{center}
$CA^\# y^{1}+(D-CA^\# B)x^{2}=y^{2}$. 
\end{center}
Premultiplying the first equation by $A$, we have $AA^\# y^{1}-Bx^{2}=Ax^{1}$ so that $Ax^{1}+Bx^{2}=AA^\# y^{1}$. Again, premultiplying the first equation by $C$, we get $CA^\# y^{1}-CA^\# Bx^{2}=Cx^{1}$. Hence, using the second equation we have, $Cx^{1}+Dx^{2}=CA^\# y^{1}-CA^\# Bx^{2}+Dx^{2}=y^{2}$, proving that 
\begin{center}
$M
\left(\begin{array}{cc}
x^{1}\\
x^{2}
\end{array}\right)
=
\left(\begin{array}{cc}
AA^\# y^{1}\\
y^{2}
\end{array}\right)$.
\end{center}
\end{proof}


\begin{thebibliography}{1}

\bibitem{bens}
A. Ben-Israel and T.N.E. Greville, {\it Generalized Inverses: Theory and Applications}: Springer-Verlag, New York, 2003.
\bibitem{kavi}
K. Bisht, G. Ravindran and K.C. Sivakumar, \emph{Pseudo Schur complements, pseudo principal pivot transforms and their inheritance properties}, Electron. J. Linear Algebra, {\bf 30}, (2015), 455-477.
\bibitem{carlhaymark}
D. Carlson, E.V. Haynsworth and T.L. Markham, {\it A generalization of the Schur complement by means of the Moore-Penrose inverse}, SIAM J. Appl. Math, {\bf 26} (1974) 169-175.
\bibitem{kaviarx}
Kavita Bisht and K.C. Sivakumar, \emph{Pseudo Schur Complements, Pseudo Principal Pivot Transforms and Their Inheritance Properties}, (April 2015) arXiv:1504.04527, 8 pp.
\bibitem{meen}
A.R. Meenakshi, {\it Principal pivot transforms of an $EP$ matrix}, C.R. Math. Rep. Acad. Sci. Canada, {\bf8} (1986) 121-126.
\bibitem{tsat}
M. Tsatsomeros, {\it Principal pivot transforms: Properties and applications}, Lin. Alg. Appl., {\bf 307} (2000), 151-165.
\bibitem{zh}
F. Zhang, {\it The Schur Complement and Its Applications}, Springer, New York, 2005.
\end{thebibliography}
\end{document}